\documentclass[12pt,reqno]{article}

\usepackage[usenames]{color}
\usepackage{amssymb}
\usepackage{amsmath}
\usepackage{amsthm}
\usepackage{amsfonts}
\usepackage{amscd}
\usepackage{graphicx}
\usepackage{mathtools}

\mathtoolsset{showonlyrefs}

\usepackage[colorlinks=true,
linkcolor=webgreen,
filecolor=webbrown,
citecolor=webgreen]{hyperref}

\definecolor{webgreen}{rgb}{0,.5,0}
\definecolor{webbrown}{rgb}{.6,0,0}

\usepackage{color}
\usepackage{fullpage}
\usepackage{float}

\usepackage{graphics}
\usepackage{latexsym}

\begin{document}

\theoremstyle{plain}
\newtheorem{theorem}{Theorem}
\newtheorem{corollary}[theorem]{Corollary}
\newtheorem{proposition}{Proposition}
\newtheorem{lemma}{Lemma}
\newtheorem{example}{Example}
\newtheorem{remark}{Remark}


\newcommand{\braces}{\genfrac{\lbrace}{\rbrace}{0pt}{}}

\begin{center}

\vskip 1cm{\large\bf  
On Series Involving Cubed Catalan Numbers}
\vskip 1cm
\large
Kunle Adegoke \\ 
Department of Physics and Engineering Physics\\
Obafemi Awolowo University\\
220005 Ile-Ife \\
Nigeria \\
\href{mailto:adegoke00@gmail.com }{\tt adegoke00@gmail.com}\\

\end{center}

\vskip .2 in

\begin{abstract}
Using generalized binomial coefficient identities and some results of John Dougall, we derive some families of series involving the cubes of Catalan numbers. We also establish a family of series containing fourth powers of Catalan numbers. Finally, we find a generalization of the Bauer series for $1/\pi$ and obtain some Ramanujan-like series for $1/\pi^2$ and~$1/\pi^3$.
\end{abstract}



\section{Introduction}

Catalan numbers are defined for non-negative integers $j$ by
\begin{equation}\label{catalan}
C_j=\frac1{j+1}\binom{2j}j,
\end{equation}
and obey the following recurrence relation:
\begin{equation}\label{qh1pgh8}
C_{j+1}=\frac{2(2j+1)}{j+2}C_j,\quad C_0=1.
\end{equation}

Using Dixon's theorem, Tauraso~\cite{tauraso20} found the following identity involving cubed Catalan numbers:
\begin{equation}\label{tauraso}
\sum_{k = 0}^\infty  {\left( {\frac{{C_k }}{{2^{2k} }}} \right)^3 }  = 8 - \frac{{384\pi }}{{\Gamma ^4 \left( {\frac{1}{4}} \right)}},
\end{equation}
where $\Gamma(z)$ is the Gamma function.

In this paper, using identities involving generalized binomial coefficients and results due to Dougall~\cite{dougall06} we will derive~\eqref{tauraso} and the following series having a similar nature:
\begin{align}
\sum_{k = 0}^\infty  {\left( {\frac{{C_k }}{{2^{2k} \left( {k + 2} \right)}}} \right)^3 }  &= \frac{{152}}{{27}} - \frac{{80}}{{81\pi ^3 }}\Gamma ^4 \left( {\frac{1}{4}} \right)\label{mb1w1rv},\\
\sum_{k = 0}^\infty  {\left( {\frac{{( - 1)^k C_k }}{{2^{2k} }}} \right)^3 \left( {4k + 3} \right)}  &= 8 - \frac{{16}}{\pi }\label{ouf2lej},\\
\sum_{k = 0}^\infty  {\left( {\frac{{( - 1)^k C_k }}{{2^{2k} \left( {k + 2} \right)}}} \right)^3 \left( {4k + 5} \right)}  &=  - \frac{8}{9} + \frac{{128}}{{27\pi }},\label{ywjgwul}\\
\sum_{k = 0}^\infty  {\left( {\frac{{C_k }}{{2^{2k} }}} \right)^3 \left( {4k + 3} \right)}  &=  - 8 + \frac{2}{{\pi ^3 }}\Gamma ^4 \left( {\frac{1}{4}} \right)\label{grwb990},
\end{align}
and
\begin{equation}\label{bsx1t97}
\sum_{k = 0}^\infty  {\left( {\frac{{C_k }}{{2^{2k} \left( {k + 2} \right)}}} \right)^3 \left( {4k + 5} \right)}  = \frac{{136}}{9} - \frac{{7168}}{9}\frac{\pi }{{\Gamma ^4 \left( {\frac{1}{4}} \right)}}.
\end{equation}
In fact, each of identities~\eqref{tauraso}--\eqref{bsx1t97} occurs as a particular member of a family of series stated in Corollaries~\ref{ogftiro} and~\ref{hk7otxz} and Theorem~\ref{c5oiueo}. For instance,~\eqref{tauraso} is the $m=0$ case of the family stated in Corollary~\ref{ogftiro}, namely,
\begin{align}
&\sum_{k = 0}^\infty  {\left(\frac{{C_k }}{{2^{2k} }}\prod_{j = 1}^{2m} {\frac{1}{{2k - 2j + 1}}} \right)^3}\\
&\qquad  = \frac{{2^{6m + 6} }}{{\left( {\left( {2m + 2} \right)!} \right)^3 C_{2m + 1}^3 }}\left( {1 - \frac{{48}}{{\left( {4m + 1} \right)^2 }}( - 1)^m \frac{\pi }{{\Gamma ^4 \left( {\frac{1}{4}} \right)}}\frac{{\prod_{j = 1}^{3m} {\left( {4j - 1} \right)} }}{{\prod_{j = 1}^m {\left( {4j - 3} \right)^3 } }}} \right).
\end{align}

We will also derive families of series involving cubed Catalan numbers and odd harmonic numbers including the following special cases:
\begin{align}
\sum_{k = 0}^\infty  {\left( {\frac{{C_k }}{{2^{2k} }}} \right)^3 O_k }  &= 8 + \frac{{64\pi \left( {\pi  - 10} \right)}}{{\Gamma ^4 \left( {\frac{1}{4}} \right)}}\label{lsg8qx3},\\
\sum_{k = 0}^\infty  {\left( {\frac{{C_k }}{{2^{2k} \left( {k + 2} \right)}}} \right)^3 O_k }  &= \frac{{392}}{{81}} - \frac{8}{{729}}\frac{{\left( {15\pi  + 32} \right)}}{{\pi ^3 }}\Gamma ^4 \left( {\frac{1}{4}} \right),\label{sce0v5v}\\
\sum_{k = 0}^\infty  {\left( {\frac{{( - 1)^k C_k }}{{2^{2k} }}} \right)^3 \left( {\left( {4k + 3} \right)O_k  + \frac{1}{3}} \right)}  &= \frac{{16\left( {\pi  - 3} \right)}}{{3\pi }}\label{vw47d00},
\end{align}
and
\begin{equation}\label{x8fwzqc}
\sum_{k = 0}^\infty  {\left( {\frac{{( - 1)^k C_k }}{{2^{2k} \left( {k + 2} \right)}}} \right)^3 \left( {\left( {4k + 5} \right)O_k  + \frac{1}{3}} \right)}  = \frac{{32\left( {16 - 5\pi } \right)}}{{81\pi }}.
\end{equation}
As additional results, we will derive a family of series involving fourth powers of Catalan numbers and one involving these numbers and odd harmonic numbers, of which the following are particular cases:
\begin{align}
\sum_{k = 0}^\infty  {\left( {\frac{{C_k }}{{2^{2k} }}} \right)^4 \left( {4k + 3} \right)}  &= 16 - \frac{{128}}{{\pi ^2 }}\label{nyc2b6c},\\
\sum_{k = 0}^\infty  {\left( {\frac{{C_k }}{{2^{2k} \left( {k + 2} \right)}}} \right)^4 \left( {4k + 5} \right)}  &=  - \frac{{176}}{{27}} + \frac{{16384}}{{243\pi ^2 }}\label{n6owsry},\\
\sum_{k = 0}^\infty  {\left( {\frac{{C_k }}{{2^{2k} }}} \right)^4 \left( {\left( {4k + 3} \right)O_k  + \frac{1}{4}} \right)}  &= 12 - \frac{{192}}{{\pi ^2 }} + \frac{{32\left( {2\ln 2 + 1} \right)}}{{\pi ^2 }}\label{gm5thpa},
\end{align}
and
\begin{equation}\label{psoxn08}
\sum_{k = 0}^\infty  {\left( {\frac{{C_k }}{{2^{2k} \left( {k + 2} \right)}}} \right)^4 \left( {\left( {4k + 5} \right)O_k  + \frac{1}{4}} \right)}  =  - \frac{{220}}{{27}} + \frac{{75776}}{{729\pi ^2 }} - \frac{{8192}}{{243}}\frac{{\ln 2}}{{\pi ^2 }}.
\end{equation}
A high point of this paper is the discovery of a family of Ramanujan-like series (Theorem~\ref{bauer_gen}):
\begin{equation}
\sum_{k = 0}^\infty  {\left( {\frac{{( - 1)^k }}{{2^{2k} }}\binom{{2k}}{k}\prod_{j = 1}^m {\frac{1}{{2k - 2j + 1}}} } \right)^3 \left( {4k - 2m + 1} \right)}  = \left( {\frac{{m!}}{{2^m }}\binom{{2m}}{m}} \right)^{ - 3} \frac{2}{\pi },\quad m=0,1,2,\ldots,
\end{equation}
which gives the Bauer series~\cite{bauer} at $m=0$:
\begin{equation}
\sum_{k = 0}^\infty  {\left( {\frac{{( - 1)^k }}{{2^{2k} }}\binom{{2k}}{k}} \right)^3 \left( {4k + 1} \right)}  = \frac{2}{\pi }.
\end{equation}
We will establish the following families of Ramanujan-like series for $1/\pi^2$ and $1/\pi^3$:
\begin{equation}
\sum_{k = 0}^\infty  {\left( {\frac{1}{{2^{2k} }}\binom{{2k}}{k}\prod_{j = 1}^m {\frac{1}{{2k - 2j + 1}}} } \right)^4 \left( {4k - 2m + 1} \right)}  = \frac{{( - 1)^m 2^{8m} }}{{(m!)^4 \pi ^2 m}}\binom{{2m}}{m}^{ - 5},\quad m=1,2,3,\ldots ,
\end{equation}
and
\begin{align}
&\sum_{k = 0}^\infty  {\left( {\frac{1}{{2^{2k} }}\binom{{2k}}{k}\prod_{j = 1}^{2m} {\frac{1}{{2k - 2j + 1}}} } \right)^3 }\\
&\qquad  = \left( {\frac{{(2m)!}}{{2^{2m} }}\binom{{4m}}{{2m}}} \right)^{ - 3} \frac{{( - 1)^m }}{{\left( {4m - 1} \right)^2 }}\frac{{\Gamma ^4 (1/4)}}{{4\pi ^3 }}\frac{{\prod_{j = 0}^{3m - 1} {\left( {4j - 3} \right)} }}{{\prod_{j = 0}^{m - 1} {\left( {4j - 1} \right)^3 } }},\quad m=0,1,2,3,\ldots;
\end{align}
with the special cases
\begin{equation}
\sum_{k = 0}^\infty  {\left( {\frac{1}{{2^{2k} \left( {2k - 1} \right)}}\binom{{2k}}{k}} \right)^4 \left( {4k - 1} \right)}  =  - \frac{8}{{\pi ^2 }}
\end{equation}
and
\begin{equation}
\sum_{k = 0}^\infty  {\left( {\frac{1}{{2^{2k} }}\binom{{2k}}{k}} \right)^3 }  = \frac{{\Gamma ^4 \left( {1/4} \right)}}{{4\pi ^3 }}.
\end{equation}
Binomial coefficients are defined, for non-negative integers $i$ and $j$, by
\begin{equation}
\binom ij=
\begin{cases}
\dfrac{{i!}}{{j!(i - j)!}}, & \text{$i \ge j$};\\
0, & \text{$i<j$}.
\end{cases}
\end{equation}
The definition is extended to complex numbers $r$ and $s$ by
\begin{equation}\label{y89d722}
\binom rs= \frac{{\Gamma (r + 1)}}{{\Gamma (s + 1)\Gamma (r - s + 1)}},
\end{equation}
where the Gamma function, $\Gamma(z)$, is defined for $\Re(z)>0$ by
\begin{equation}
\Gamma (z) = \int_0^\infty  {e^{ - t} t^{z - 1}dt}  = \int_0^\infty  {\left( {\log (1/t)} \right)^{z - 1}dt},
\end{equation}
and is extended to the rest of the complex plane, excluding the non-positive integers, by analytic continuation.

Harmonic numbers, $H_j$, and odd harmonic numbers, $O_j$, are defined for non-negative integers $j$ by
\begin{equation}
H_j=\sum_{k=1}^j\frac1k,\quad O_j=\sum_{k=1}^j\frac1{2k-1},\quad H_0=0=O_0.
\end{equation}
The identity
\begin{equation}\label{nerd5x1}
H_{n-1/2}=2O_n-2\ln 2,
\end{equation}
extends harmonic numbers to half-integer arguments and is a consequence of the link between harmonic numbers, odd harmonic numbers and the digamma function.

\section{Required results}
Series involving cubed Catalan numbers are given in Theorems~\ref{ldota2q},~\ref{yn6aprj} and~\ref{c5oiueo} and Corollaries~\ref{ogftiro} and~\ref{hk7otxz}. We require the results stated in Lemmata~\ref{w00yljy}--\ref{cif3lvo}.
\begin{lemma}\label{w00yljy}
If $k$ and $r$ are non-negative integers, then
\begin{equation}\label{chgudxy}
\binom{{r + 1/2}}{{k + 1}} = ( - 1)^r \frac{{\left( {r + 2} \right)!C_{r + 1} }}{{2^{r + 2} }}( - 1)^k \frac{{C_k }}{{2^{2k} }}\prod_{j = 1}^r {\frac{1}{{2k - 2j + 1}}} .
\end{equation}
\end{lemma}
\begin{proof}
This identity and similar ones can be derived using~\eqref{y89d722} together with the identities
\begin{equation}\label{v7g54vl}
\Gamma \left( {z + \frac{1}{2}} \right) = \sqrt \pi\, 2^{-2z}\binom{{2z}}{z}\,\Gamma \left( {z + 1} \right)
\end{equation}
and
\begin{equation}\label{ho6212z}
\Gamma \left( { - z + \frac{1}{2}} \right) = ( - 1)^z\, 2^{2z} \,\binom{{2z}}{z}^{ - 1} \frac{{\sqrt \pi  }}{{\Gamma \left( {z + 1} \right)}},
\end{equation}
and the Pochhammer relation
\begin{equation}\label{jctzn3r}
\frac{{\Gamma \left( {z + n} \right)}}{{\Gamma \left( z \right)}} =\prod_{j=1}^n{(z+j-1)}=(z)_n,
\end{equation}
for $n$ a non-negative integer and $z$ a complex number.
\end{proof}
\begin{lemma}\label{prlbrni}
If $m$ is an integer, then
\begin{equation}
\frac{{\Gamma \left( {\frac{1}{4}\left( {6m + 7} \right)} \right)}}{{\Gamma ^3 \left( {\frac{1}{4}\left( {2m + 5} \right)} \right)}} = \frac{{48}}{{\left( {2m + 1} \right)^2 }}\,\frac{{\Gamma \left( {\frac{3}{4}\left( {2m + 1} \right)} \right)}}{{\Gamma ^3 \left( {\frac{1}{4}\left( {2m + 1} \right)} \right)}}
\end{equation}
and
\begin{equation}
\frac{{\Gamma \left( {\frac{1}{4}\left( {6m + 5} \right)} \right)}}{{\Gamma ^3 \left( {\frac{1}{4}\left( {2m + 3} \right)} \right)}} = \frac{{12\left( {6m + 1} \right)}}{{\left( {2m - 1} \right)^2 }}\frac{{\Gamma \left( {\frac{3}{4}\left( {2m - 1} \right)} \right)}}{{\Gamma ^3 \left( {\frac{1}{4}\left( {2m - 1} \right)} \right)}}.
\end{equation}
\end{lemma}
\begin{proof}
Use the recurrence relation $\Gamma(u)=\left(u-1\right)\,\Gamma(u-1)$.
\end{proof}
\begin{lemma}\label{x2m9k5r}
If $m$ is an integer, then
\begin{equation}
\Gamma \left( {m + \frac{1}{4}} \right) = \frac{1}{{4^m }}\,\Gamma \left( {\frac{1}{4}} \right)\,\prod\limits_{j = 1}^m {\left( {4j - 3} \right)}
\end{equation}
and
\begin{equation}
\Gamma \left( {m + \frac{3}{4}} \right) = \frac{1}{{4^m }}\,\frac{{\pi \sqrt 2 }}{{\Gamma \left( {\frac{1}{4}} \right)}}\,\prod\limits_{j = 1}^m {\left( {4j - 1} \right)}.
\end{equation}
\end{lemma}
\begin{proof}
These follow immediately from~\eqref{jctzn3r} and the fact that
\begin{equation}
\Gamma \left(\frac34\right)=\frac{\pi\sqrt 2}{\Gamma \left(\frac14\right)},
\end{equation}
which is a consequence of Euler's reflection formula:
\begin{equation}
\Gamma(z)\,\Gamma(1-z)=\frac\pi{\sin(\pi z)}.
\end{equation}
\end{proof}
\begin{lemma}\label{m8epvg7}
If $m$ is an integer, then
\begin{align}
\cos \left( {\frac{\pi }{4}\left( {2m + 1} \right)} \right) &= \frac{{( - 1)^{\left\lceil {m/2} \right\rceil } }}{{\sqrt 2 }}\label{t6ka1s8},\\
\sin \left( {\frac{\pi }{4}\left( {2m + 1} \right)} \right) &= \frac{{( - 1)^{m + \left\lceil {m/2} \right\rceil } }}{{\sqrt 2 }},\\
\cos \left( {\frac{\pi }{4}\left( {2m - 1} \right)} \right) &= \frac{{( - 1)^{\left\lfloor {m/2} \right\rfloor } }}{{\sqrt 2 }},
\end{align}
and
\begin{equation}
\sin \left( {\frac{\pi }{4}\left( {2m - 1} \right)} \right) =  - \frac{{( - 1)^{\left\lceil {m/2} \right\rceil } }}{{\sqrt 2 }};
\end{equation}
where, as usual, $\left\lfloor x \right\rfloor$ is the greatest integer less than or equal to $x$ and $\left\lceil x \right\rceil$ is the smallest integer greater than or equal to $x$.
\end{lemma}
\begin{proof}
We prove~\eqref{t6ka1s8}. We have
\begin{align}
\cos \left( {\frac{\pi }{4}\left( {2m + 1} \right)} \right) &= \cos \left( {\frac{{\pi m}}{2}} \right)\cos \left( {\frac{\pi }{4}} \right) - \sin \left( {\frac{{\pi m}}{2}} \right)\sin \left( {\frac{\pi }{4}} \right)\\
&= \frac{1}{{\sqrt 2 }}\left( {\cos \left( {\frac{{\pi m}}{2}} \right) - \sin \left( {\frac{{\pi m}}{2}} \right)} \right)\\
&= \frac{1}{{\sqrt 2 }} 
\begin{cases}
 \cos \left( {\dfrac{{\pi m}}{2}} \right),&\text{if $m$ is even;} \\ 
  - \sin \left( {\dfrac{{\pi m}}{2}} \right),&\text{if $m$ is odd.} \\ 
 \end{cases}
\end{align}
Thus,
\begin{align}
\cos \left( {\frac{\pi }{4}\left( {2m + 1} \right)} \right) &= \frac{1}{{\sqrt 2 }}
\begin{cases}
 ( - 1)^{m/2},&\text{if $m$ is even;}  \\ 
 ( - 1)^{(m + 1)/2},&\text{if $m$ is odd.}  \\ 
 \end{cases}\\
& = \frac{{( - 1)^{\left\lceil {m/2} \right\rceil } }}{{\sqrt 2 }}.
\end{align}
\end{proof}
\begin{lemma}
If $n$ is an integer, then $O_{-n}=O_n$.
\end{lemma}
\begin{proof}
We have
\begin{equation}
O_{ - n}  = \sum_{j = 1}^{ - n} {\frac{1}{{2j - 1}}}  =  - \sum_{j = 1 - n}^0 {\frac{1}{{2j - 1}}}  =  - \sum_{j = 1}^n {\frac{1}{{2j - 2n - 1}}} .
\end{equation}
Thus,
\begin{equation}
O_{ - n}  = \sum_{j = 1}^n {\frac{1}{{2n - 2j + 1}}}  = \sum_{j = 1}^n {\frac{1}{{2n - 2\left( {n - j + 1} \right) + 1}}}  = \sum_{j = 1}^n {\frac{1}{{2j - 1}}}  = O_n .
\end{equation}
\end{proof}
Each identity in Lemma~\ref{cif3lvo} is a slight re-writing of the original one derived by Dougall~\cite{dougall06}.
\begin{lemma}\label{cif3lvo}
If $x$ is a complex number, then
\begin{align}
\sum_{k = 0}^\infty  {( - 1)^k \binom{{x}}{{k + 1}}^3 }  &= 1 - \frac{{\Gamma \left( {\frac{1}{2}\left( {3x + 2} \right)} \right)}}{{\Gamma ^3 \left( {\frac{1}{2}\left( {x + 2} \right)} \right)}}\cos \left( {\frac{{\pi x}}{2}} \right)\label{gfrhf9n},\quad\Re\, x>-2/3,\\
\sum_{k = 0}^\infty  {\binom{{x}}{{k + 1}}^3 \left( {2k  + 2 - x} \right)}  &= x - \frac{{\sin (\pi x)}}{\pi },\quad x\ge -1/3\label{g2cbohu},
\end{align}
and
\begin{equation}\label{bxmgxsj}
\sum_{k = 0}^\infty  {( - 1)^k \binom{{x}}{{k + 1}}^3 \left( {2k + 2 - x} \right)}  =  - x + \frac{{\sin \left( {\pi x} \right)}}{\pi }\frac{{\Gamma \left( {\frac{1}{2}\left( {1 - x} \right)} \right)\Gamma \left( {\frac{1}{2}\left( {1 + 3x} \right)} \right)}}{{\Gamma ^2 \left( {\frac{1}{2}\left( {1 + x} \right)} \right)}}.
\end{equation}
\end{lemma}
\begin{remark}
Identities~\eqref{gfrhf9n}--\eqref{bxmgxsj} are also recorded in~\cite[Equations (6.8)--(6.10), p.52]{gould}.
\end{remark}
\section{Series involving cubed Catalan numbers}
\begin{theorem}\label{ldota2q}
If $m$ is a non-negative integer, then
\begin{align}
&\sum_{k = 0}^\infty  {\left( {\frac{{C_k }}{{2^{2k} }}\prod_{j = 1}^m {\frac{1}{{2k - 2j + 1}}} } \right)^3 }\\
&\qquad\qquad  = \frac{{( - 1)^m 2^{3m + 6} }}{{\left( {\left( {m + 2} \right)!} \right)^3 C_{m + 1}^3 }}\left( {1 - \frac{{24\sqrt 2 }}{{\left( {2m + 1} \right)^2 }}\,( - 1)^{\left\lceil {m/2} \right\rceil }\, \frac{{\Gamma \left( {\frac{3}{4}\left( {2m + 1} \right)} \right)}}{{\Gamma^3 \left( {\frac{1}{4}\left( {2m + 1} \right)} \right)}}} \right)\label{ghsuozj}.
\end{align}

\end{theorem}
\begin{proof}
Set $x=m+1/2$ in~\eqref{gfrhf9n} to obtain
\begin{equation}
\sum_{k = 0}^\infty  {( - 1)^k \binom{{m + 1/2}}{{k + 1}}^3 }  = 1 - \frac{{\Gamma \left( {\frac{1}{4}\left( {6m + 7} \right)} \right)}}{{\Gamma ^3 \left( {\frac{1}{4}\left( {2m + 5} \right)} \right)}}\cos \left( {\frac{\pi }{4}\left( {2m + 1} \right)} \right),
\end{equation}
and hence~\eqref{ghsuozj} upon using Lemmata~\ref{w00yljy},~\ref{prlbrni} and~\ref{m8epvg7}.
\end{proof}
\begin{corollary}\label{ogftiro}
If $m$ is a non-negative integer, then
\begin{align}
&\sum_{k = 0}^\infty  {\left(\frac{{C_k }}{{2^{2k} }}\prod_{j = 1}^{2m} {\frac{1}{{2k - 2j + 1}}} \right)^3}\\
&\qquad  = \frac{{2^{6m + 6} }}{{\left( {\left( {2m + 2} \right)!} \right)^3 C_{2m + 1}^3 }}\left( {1 - \frac{{48}}{{\left( {4m + 1} \right)^2 }}( - 1)^m \frac{\pi }{{\Gamma ^4 \left( {\frac{1}{4}} \right)}}\frac{{\prod_{j = 1}^{3m} {\left( {4j - 1} \right)} }}{{\prod_{j = 1}^m {\left( {4j - 3} \right)^3 } }}} \right)\label{wp90c4e}
\end{align}
and
\begin{align}
&\sum_{k = 0}^\infty  {\left(\frac{{C_k }}{{2^{2k} }}\prod_{j = 1}^{2m - 1} {\frac{1}{{2k - 2j + 1}}} \right)^3}\\
&\qquad  =  - \frac{{2^{6m + 3} }}{{\left( {\left( {2m + 1} \right)!} \right)^3 C_{2m}^3 }}\left( {1 - \frac{{( - 1)^m }}{{\left( {4m - 1} \right)^2 }}\,\frac{3}{{4\pi ^3 }}\,\Gamma ^4 \left( {\frac{1}{4}} \right)\frac{{\prod_{j = 1}^{3m - 1} {\left( {4j - 3} \right)} }}{{\prod_{j = 1}^{m - 1} {\left( {4j - 1} \right)^3 } }}} \right).\label{d060e1c}
\end{align}
\end{corollary}
\begin{proof}
Write $2m$ for $m$ and $2m-1$ for $m$, in turn, in~\eqref{ghsuozj}, and apply Lemma~\ref{x2m9k5r}.
\end{proof}
\begin{remark}
Identities~\eqref{tauraso} and~\eqref{mb1w1rv} are obtained by evaluating~\eqref{wp90c4e} at $m=0$ and~\eqref{d060e1c} at $m=1$.
\end{remark}
\begin{theorem}\label{yn6aprj}
If $m$ is a non-negative integer, then
\begin{align}
&\sum_{k = 0}^\infty  {( - 1)^k \left( {\frac{{C_k }}{{2^{2k} }}\prod_{j = 1}^m {\frac{1}{{2k - 2j + 1}}} } \right)^3 } \left( {4k - 2m + 3} \right)\\
&\qquad\qquad = \frac{{( - 1)^m 2^{3m + 6} }}{{\left( {\left( {m + 2} \right)!} \right)^3 C_{m + 1}^3 }}\left( {2m + 1 - ( - 1)^m \frac{2}{\pi }} \right)\label{aanj1h6}.
\end{align}

\end{theorem}

\begin{proof}
Use of $r=m+1/2$ in~\eqref{g2cbohu} gives
\begin{equation}
\sum_{k = 0}^\infty  {\binom{{m + 1/2}}{{k + 1}}^3 \left( {4k - 2m + 3} \right)}  = 2m + 1 + ( - 1)^{m + 1} \frac{2}{\pi },
\end{equation}
and hence~\eqref{aanj1h6} by Lemma~\ref{w00yljy}.
\end{proof}
\begin{remark}
Identities~\eqref{ouf2lej} and~\eqref{ywjgwul} are obtained at $m=0$ and $m=1$ in~\eqref{aanj1h6}.
\end{remark}
\begin{theorem}\label{c5oiueo}
If $m$ is a non-negative integer, then
\begin{align}\label{vksliwm}
&\sum_{k = 0}^\infty  {\left( {\frac{{C_k }}{{2^{2k} }}\prod_{j = 1}^m {\frac{1}{{2k - 2j + 1}}} } \right)^3 \left( {4k - 2m + 3} \right)}\\
&\qquad  = \frac{{( - 1)^{m - 1} 2^{3m + 6} }}{{\left( {\left( {m + 2} \right)!} \right)^3 C_{m + 1}^3 }}\left( {2m + 1 - ( - 1)^{m + \left\lceil {m/2} \right\rceil } 24\sqrt 2 \frac{{6m + 1}}{{\left( {2m - 1} \right)^2 }}\frac{{\Gamma \left( {\frac{3}{4}\left( {2m - 1} \right)} \right)}}{{\Gamma^3 \left( {\frac{1}{4}\left( {2m - 1} \right)} \right) }}} \right).
\end{align}
\end{theorem}
\begin{proof}
Set $x=m+1/2$ in~\eqref{bxmgxsj} and apply Lemmata~\ref{w00yljy},~\ref{prlbrni} and~\ref{m8epvg7}.
\end{proof}
\begin{corollary}\label{hk7otxz}
If $m$ is a non-negative integer, then
\begin{align}\label{i43l2u1}
&\sum_{k = 0}^\infty  {\left( {\frac{{C_k }}{{2^{2k} }}\prod_{j = 1}^{2m} {\frac{1}{{2k - 2j + 1}}} } \right)^3 \left( {4k - 4m + 3} \right)}\\
&\qquad  =  - \frac{{2^{6m + 6} }}{{\left( {\left( {2m + 2} \right)!} \right)^3 C_{2m + 1}^3 }}\left( {4m + 1 - \frac{{( - 1)^m \left( {12m + 1} \right)}}{{\left( {4m - 1} \right)^2 }}\frac{3}{{4\pi ^3 }}\Gamma ^4 \left( {\frac{1}{4}} \right)\frac{{\prod_{j = 1}^{3m - 1} {\left( {4j - 3} \right)} }}{{\prod_{j = 1}^{m - 1} {\left( {4j - 1} \right)^3 } }}} \right)
\end{align}
and
\begin{align}\label{lp9cchn}
&\sum_{k = 0}^\infty  {\left( {\frac{{C_k }}{{2^{2k} }}\prod_{j = 1}^{2m - 1} {\frac{1}{{2k - 2j + 1}}} } \right)^3 \left( {4k - 4m + 5} \right)}\\
&\qquad  = \frac{{2^{6m + 3} }}{{\left( {\left( {2m + 1} \right)!} \right)^3 C_{2m}^3 }}\left( {4m - 1 + \frac{{( - 1)^m \left( {12m - 5} \right)}}{{\left( {4m - 3} \right)^2 }}\frac{{48\pi }}{{\Gamma ^4 \left( {\frac{1}{4}} \right)}}\frac{{\prod_{j = 1}^{3m - 3} {\left( {4j - 1} \right)} }}{{\prod_{j = 1}^{m - 1} {\left( {4j - 3} \right)^3 } }}} \right).
\end{align}
\end{corollary}
\begin{proof}
Write $2m$ for $m$ and $2m-1$ for $m$, in turn, in~\eqref{vksliwm}, and apply Lemma~\ref{x2m9k5r}.
\end{proof}
\begin{remark}
Identities~\eqref{grwb990} and~\eqref{bsx1t97} correspond to an evaluation of~\eqref{i43l2u1} at $m=0$ and~\eqref{lp9cchn} at $m=1$.
\end{remark}

\section{Series involving cubed Catalan numbers and odd harmonic numbers}

\begin{theorem}
If $m$ is a non-negative integer, then
\begin{align}
&\sum_{k = m}^\infty  {\left( {\frac{{C_k }}{{2^{2k} }}\prod_{j = 1}^m {\frac{1}{{2k - 2j + 1}}} } \right)^3 O_{k - m} }\\
&\qquad= -\sum_{k = 0}^m {\left( {\frac{{C_k }}{{2^{2k} }}\prod_{j = 1}^m {\frac{1}{{2k - 2j + 1}}} } \right)^3 O_{m - k} }  + \frac{{( - 1)^m 2^{3m + 6} O_{m + 1} }}{{\left( {\left( {m + 2} \right)!} \right)^3 C_{m + 1}^3 }}\\
& \qquad\qquad +\frac{{( - 1)^{m + \left\lceil {m/2} \right\rceil } 2^{3m + 7} }}{{\left( {\left( {m + 2} \right)!} \right)^3 C_{m + 1}^3 }}\frac{{3\sqrt 2 }}{{\left( {2m + 1} \right)^2 }}\frac{{\Gamma \left( {\frac{3}{4}\left( {2m + 1} \right)} \right)}}{{\Gamma \left( {\frac{1}{4}\left( {2m + 1} \right)} \right)^3 }}\\
&\qquad\qquad\qquad\times\left( {( - 1)^{m - 1} \frac{\pi }{3} + H_{\frac{{3\left( {2m + 1} \right)}}{4}}  - H_{\frac{{\left( {2m + 1} \right)}}{4}}  - 4O_{m + 1} } \right)\label{szidc7y}.
\end{align}
\end{theorem}
\begin{proof}
Write~\eqref{gfrhf9n} in the equivalent form
\begin{equation}
\sum_{k = 0}^\infty  {( - 1)^k \binom{{x}}{{k + 1}}^3 }  = 1 - \binom{{3x/2}}{{x/2}}\binom{{x}}{{x/2}}\cos \left( {\frac{{\pi x}}{2}} \right),
\end{equation}
and differentiate with respect to $x$ to obtain
\begin{align}
&3\sum_{k = 0}^\infty  {( - 1)^k \binom{{x}}{{k + 1}}^3 \left( {H_x  - H_{x - k - 1} } \right)}\\
&\qquad  = \left( {\frac{\pi }{2}\sin \left( {\frac{{\pi x}}{2}} \right) - \left( {\frac{3}{2}H_{3x/2}  - \frac{3}{2}H_{x/2} } \right)\cos \left( {\frac{{\pi x}}{2}} \right)} \right)\frac{{\Gamma \left( {\frac{1}{2}\left( {3x + 2} \right)} \right)}}{{\Gamma ^3 \left( {\frac{1}{2}\left( {x + 2} \right)} \right)}}.
\end{align}
Set $x=m+1/2$ and use~\eqref{nerd5x1} and Lemma~\ref{m8epvg7} to get
\begin{align}
&\sum_{k = 0}^\infty  {( - 1)^k \binom{{m + 1/2}}{{k + 1}}^3 O_{k - m} }\\
&\qquad  = \left( {( - 1)^{m - 1} \frac{\pi }{3} + H_{\frac{3}{4}\left( {2m + 1} \right)}  - H_{\frac{1}{4}\left( {2m + 1} \right)}  - 4O_{m + 1} } \right)\frac{{( - 1)^{\left\lceil {m/2} \right\rceil } }}{{4\sqrt 2 }}\frac{{\Gamma \left( {\frac{1}{4}\left( {6m + 7} \right)} \right)}}{{\Gamma ^3 \left( {\frac{1}{4}\left( {2m + 5} \right)} \right)}}.
\end{align}
Identity~\eqref{szidc7y} now follows upon application of Lemmata~\ref{w00yljy} and~\ref{prlbrni}. 
\end{proof}
\begin{remark}
Identities~\eqref{lsg8qx3} and~\eqref{sce0v5v} correspond to evaluating~\eqref{szidc7y} at $m=0$ and $m=1$. In deriving~\eqref{lsg8qx3}, note that
\begin{equation}
H_{3/4}  - H_{1/4}  = \frac{1}{{3/4}} - \frac{1}{{1/4}} + \pi \cot \left( {\frac{\pi }{4}} \right) = \pi  - \frac{8}{3},
\end{equation}
since
\begin{equation}
H_r  - H_{1 - r}  = \frac{1}{r} - \frac{1}{{1 - r}} + \pi \cot \left( {\pi \left( {1 - r} \right)} \right),
\end{equation}
when $r$ is not an integer. For~\eqref{sce0v5v}, note also that
\begin{equation}
H_{9/4}  = H_{5/4}  + \frac{4}{9} = H_{1/4}  + \frac{4}{5} + \frac{4}{9} = H_{1/4}  + \frac{{56}}{{45}},
\end{equation}
since
\begin{equation}
H_r=H_{r-1}+\frac1r,
\end{equation}
when $r$ is not a non-positive integer. Thus
\begin{equation}
H_{9/4}  - H_{3/4}  = H_{1/4}  - H_{3/4}  + \frac{{56}}{{45}} = \frac{8}{3} - \pi  + \frac{{56}}{{45}} = \frac{{176}}{{45}} - \pi.
\end{equation}
\end{remark}
\begin{theorem}
If $m$ is a non-negative integer, then
\begin{align}
&\sum_{k = m}^\infty  {\left( {\frac{{( - 1)^k C_k }}{{2^{2k} }}\prod_{j = 1}^m {\frac{1}{{2k - 2j + 1}}} } \right)^3 \left( {\left( {4k - 2m + 3} \right)O_{k - m}  + \frac{1}{3}} \right)}\\
&\qquad=  - \sum_{k = 0}^{m - 1} {\left( {\frac{{( - 1)^k C_k }}{{2^{2k} }}\prod_{j = 1}^m {\frac{1}{{2k - 2j + 1}}} } \right)^3 \left( {\left( {4k - 2m + 3} \right)O_{m - k}  + \frac{1}{3}} \right)}\\
&\qquad\qquad+ \frac{{( - 1)^m 2^{3m + 6} }}{{\left( {\left( {m + 2} \right)!} \right)^3 C_{m + 1}^3 }}\left( {\left( {2m + 1 - \frac{{( - 1)^m 2}}{\pi }} \right)O_{m + 1}  - \frac{1}{3}} \right)\label{kgkn8io}.
\end{align}
\end{theorem}
\begin{proof}
Differentiate~\eqref{g2cbohu} with respect to $x$ to obtain
\begin{align}
\sum_{k = 0}^\infty  {\binom{{x}}{{k + 1}}^3 \left( {\left( {2k + 2 - x} \right)H_{x - k - 1}  + \frac{1}{3}} \right)}  = H_x \left( {x - \frac{{\sin (\pi x)}}{\pi }} \right) - \frac{1}{3} + \frac{{\cos (\pi x)}}{3}.
\end{align}
Set $x=m+1/2$ and use~\eqref{nerd5x1} and Lemma~\ref{w00yljy}.
\end{proof}
\begin{remark}
Identities~\eqref{vw47d00} and~\eqref{x8fwzqc} are evaluations of~\eqref{kgkn8io} at $m=0$ and $m=1$.
\end{remark}

\section{Series involving fourth powers of Catalan numbers}
In this section, we derive a family of series involving fourth powers of Catalan numbers, based on the identity
\begin{equation}\label{c4hfyci}
\sum_{k = 0}^\infty  {\binom{{x}}{{k + 1}}^4 \left( {2k - x + 2} \right)}  = x - \frac{{\sin (\pi x)}}{\pi }\binom{{2x}}{x},\quad x>-\frac12,
\end{equation}
which is a variation on~\cite[Equation (17)]{dougall06}.
\begin{theorem}
If $m$ is a non-negative integer, then
\begin{align}
&\sum_{k = 0}^\infty  {\left( {\frac{{C_k }}{{2^{2k} }}\prod_{j = 1}^m {\frac{1}{{2k - 2j + 1}}} } \right)^4 \left( {4k - 2m + 3} \right)}\\
&\qquad  = \frac{{2^{4m + 8} }}{{\left( {\left( {m + 2} \right)!} \right)^4 C_{m + 1}^4 }}\left( {2m + 1 - \frac{{( - 1)^m 2^{4m + 3} }}{{\pi ^2 \left( {m + 1} \right)\left( {2m + 1} \right)C_m }}} \right)\label{b3wbl9o}.
\end{align}
\end{theorem}
\begin{proof}
At $x=m+1/2$,~\eqref{c4hfyci} gives
\begin{equation}
\sum_{k = 0}^\infty  {\binom{{m + 1/2}}{{k + 1}}^4 \left( {4k - 2m + 3} \right)}  = 2m + 1 - ( - 1)^m\frac{{4 }}{\pi }\binom{{2m}}{{m + 1/2}},
\end{equation}
from which~\eqref{b3wbl9o} follows by~\eqref{chgudxy} and the fact that
\begin{equation}\label{pbes5gk}
\binom{{2m}}{{m + 1/2}} = \frac{{\Gamma \left( {2m + 1} \right)}}{{\Gamma \left( {m + 3/2} \right)\Gamma \left( {m + 1/2} \right)}} = \frac{{2^{4m + 1} }}{{\pi \left( {m + 1} \right)\left( {2m + 1} \right)C_m }}.
\end{equation}
\end{proof}
\begin{remark}
Identities~\eqref{nyc2b6c} and~\eqref{n6owsry} are obtained by evaluating~\eqref{b3wbl9o} at $m=0$ and at $m=1$.
\end{remark}
\begin{theorem}
If $m$ is a non-negative integer, then
\begin{align}
&\sum_{k = m}^\infty  {\left( {\frac{{C_k }}{{2^{2k} }}\prod_{j = 1}^m {\frac{1}{{2k - 2j + 1}}} } \right)^4 \left( {\left( {4k - 2m + 3} \right)O_{k - m}  + \frac{1}{4}} \right)}\\
&\qquad  =  - \sum_{k = 0}^{m-1} {\left( {\frac{{C_k }}{{2^{2k} }}\prod_{j = 1}^m {\frac{1}{{2k - 2j + 1}}} } \right)^4 \left( {\left( {4k - 2m + 3} \right)O_{m - k}  + \frac{1}{4}} \right)}\\
&\qquad\qquad  + \frac{{2^{4m + 8} O_{m + 1} }}{{\left( {\left( {m + 2} \right)!} \right)^4 C_{m + 1}^4 }}\left( {2m + 1 + \frac{{3( - 1)^{m + 1} 2^{4m + 2} }}{{\pi ^2 \left( {m + 1} \right)\left( {2m + 1} \right)C_m }}} \right)\\
&\qquad\qquad\quad + \frac{{2^{4m + 8} }}{{\left( {\left( {m + 2} \right)!} \right)^4 C_{m + 1}^4 }}\left( {\frac{{( - 1)^m \left( {2\ln 2 + H_{2m + 1} } \right)2^{4m + 1} }}{{\pi ^2 \left( {m + 1} \right)\left( {2m + 1} \right)C_m }} - \frac{1}{4}} \right)\label{dnhe1yg}.
\end{align}
\end{theorem}
\begin{proof}
Differentiating~\eqref{c4hfyci} with respect to $x$ gives, after some re-arrangement,
\begin{align}
&\sum_{k = 0}^\infty  {\binom{{x}}{{k + 1}}^4 \left( {\left( {2k + 2 - x} \right)H_{x - k - 1}  + \frac{1}{4}} \right)}\\
&\qquad  = \frac{{\sin (\pi x)}}{{2\pi }}\binom{{2x}}{x}\left( {H_{2x}  - 3H_x } \right) + xH_x  - \frac{1}{4} + \frac{1}{4}\binom{{2x}}{x}\cos (\pi x).
\end{align}
Evaluation at $x=m+1/2$, using~\eqref{nerd5x1} and some algebra yields
\begin{align}
&\sum_{k = 0}^\infty  {\binom{{m + 1/2}}{{k + 1}}^4 \left( {\left( {4k - 2m + 3} \right)O_{m - k}  + \frac{1}{4}} \right)}\\
&\qquad= \left( {2m + 1 + \frac{{( - 1)^{m + 1}6 }}{\pi }\binom{{2m}}{{m + 1/2}}} \right)O_{m + 1}\\
&\qquad\qquad  + \frac{{( - 1)^m }}{\pi }\left( {2\ln 2 + H_{2m + 1} } \right)\binom{{2m}}{{m + 1/2}} - \frac{1}{4}.
\end{align}
Application of~\eqref{chgudxy} and~\eqref{pbes5gk} now produces~\eqref{dnhe1yg}.
\end{proof}
\begin{remark}
Identities~\eqref{gm5thpa} and~\eqref{psoxn08} correspond to setting $m=0$ and $m=1$ in~\eqref{dnhe1yg}.
\end{remark}
\section{Ramanujan-like series and related series}
Series involving central binomial coefficients can be derived by writing~\eqref{chgudxy} in the equivalent form
\begin{equation}\label{hk5w16y}
\binom{{r - 1/2}}{k} = ( - 1)^r \frac{{r!}}{{2^r }}\binom{{2r}}{r}\frac{{( - 1)^k }}{{2^{2k} }}\binom{{2k}}{k}\prod_{j = 1}^r {\frac{1}{{2k - 2j + 1}}} ,
\end{equation}
for non-negative integers $k$ and $r$.
\begin{theorem}\label{bauer_gen}
If $m$ is a non-negative integer, then
\begin{equation}
\sum_{k = 0}^\infty  {\left( {\frac{{( - 1)^k }}{{2^{2k} }}\binom{{2k}}{k}\prod_{j = 1}^m {\frac{1}{{2k - 2j + 1}}} } \right)^3 \left( {4k - 2m + 1} \right)}  = \left( {\frac{{m!}}{{2^m }}\binom{{2m}}{m}} \right)^{ - 3} \frac{2}{\pi }.
\end{equation}
In particular,
\begin{align}
\sum_{k = 0}^\infty  {\left( {\frac{{( - 1)^k }}{{2^{2k} }}\binom{{2k}}{k}} \right)^3 \left( {4k + 1} \right)} & = \frac{2}{\pi }\label{elz3d9q},\\
\sum_{k = 0}^\infty  {\left( {\frac{{( - 1)^k }}{{2^{2k} \left( {2k - 1} \right)}}\binom{{2k}}{k}} \right)^3 \left( {4k - 1} \right)}  &= \frac{2}{\pi },
\end{align}
and
\begin{equation}
\sum_{k = 0}^\infty  {\left( {\frac{{( - 1)^k }}{{2^{2k} \left( {2k - 1} \right)\left( {2k - 3} \right)}}\binom{{2k}}{k}} \right)^3 \left( {4k - 3} \right)}  = \frac{2}{{27\pi }}.
\end{equation}
\end{theorem}
\begin{proof}
Write~\eqref{g2cbohu} as
\begin{equation}\label{hfyxpf9}
\sum_{k = 0}^\infty  {\binom xk^3 \left( {2k  - x} \right)}  = - \frac{{\sin (\pi x)}}{\pi },
\end{equation}
set $x=m-1/2$ and use~\eqref{hk5w16y}.
\end{proof}
\begin{remark}
Identity~\eqref{elz3d9q} was discovered in 1859 by Bauer~\cite{bauer} via a Fourier-Legendre expansion.
\end{remark}
\begin{theorem}
If $m$ is a non-negative integer, then
\begin{align}
&\sum_{k = m}^\infty  {\left( {\frac{{( - 1)^k }}{{2^{2k} }}\binom{{2k}}{k}\prod_{j = 1}^m {\frac{1}{{2k - 2j + 1}}} } \right)^3 \left( {\left( {4k - 2m + 1} \right)O_{k - m}  + \frac{1}{3}} \right)}\\
&\qquad  =  - \sum_{k = 0}^{m - 1} {\left( {\frac{{( - 1)^k }}{{2^{2k} }}\binom{{2k}}{k}\prod_{j = 1}^m {\frac{1}{{2k - 2j + 1}}} } \right)^3 \left( {\left( {4k - 2m + 1} \right)O_{m - k}  + \frac{1}{3}} \right)}\\
&\qquad\qquad  + \left( {\frac{{m!}}{{2^m }}\binom{{2m}}{m}} \right)^{ - 3} \frac{2}{\pi }O_m\label{sl7k8cg} .
\end{align}
In particular,
\begin{equation}
\sum_{k = 0}^\infty  {\left( {\frac{{( - 1)^k }}{{2^{2k} }}\binom{{2k}}{k}} \right)^3 \left( {4k + 1} \right)O_k }  =  - \frac{1}{3}\sum_{k = 0}^\infty  {\left( {\frac{{( - 1)^k }}{{2^{2k} }}\binom{{2k}}{k}} \right)^3 } .
\end{equation}
\end{theorem}
\begin{proof}
Differentiate~\eqref{hfyxpf9} with respect to $x$, obtaining, after some rearrangement,
\begin{equation}
\sum_{k = 0}^\infty  {\binom{{x}}{k}^3 \left( {\left( {2k - x} \right)H_{x - k}  + \frac{1}{3}} \right)}  =  - H_x \frac{{\sin (\pi x)}}{\pi } + \frac{1}{3}\cos (\pi x),
\end{equation}
and set $x=m-1/2$.
\end{proof}
\begin{remark}
Setting $m=1$ in~\eqref{sl7k8cg} reproduces~\eqref{vw47d00} after some manipulation.
\end{remark}
\begin{theorem}
If $m$ is a positive integer, then
\begin{equation}\label{ua6fod4}
\sum_{k = 0}^\infty  {\left( {\frac{1}{{2^{2k} }}\binom{{2k}}{k}\prod_{j = 1}^m {\frac{1}{{2k - 2j + 1}}} } \right)^4 \left( {4k - 2m + 1} \right)}  = \frac{{( - 1)^m 2^{8m} }}{{(m!)^4 \pi ^2 m}}\binom{{2m}}{m}^{ - 5} .
\end{equation}
In particular,
\begin{equation}
\sum_{k = 0}^\infty  {\left( {\frac{1}{{2^{2k} \left( {2k - 1} \right)}}\binom{{2k}}{k}} \right)^4 \left( {4k - 1} \right)}  =  - \frac{8}{{\pi ^2 }}
\end{equation}
and
\begin{equation}
\sum_{k = 0}^\infty  {\left( {\frac{1}{{2^{2k} \left( {2k - 1} \right)\left( {2k - 3} \right)}}\binom{{2k}}{k}} \right)^4 \left( {4k - 3} \right)}  = \frac{{64}}{{243\pi ^2 }}.
\end{equation}
\end{theorem}
\begin{proof}
Write~\eqref{c4hfyci} as
\begin{equation}
\sum_{k = 0}^\infty  {\binom{{x}}{{k}}^4 \left( {2k - x} \right)}  =- \frac{{\sin (\pi x)}}{\pi }\binom{{2x}}{x},\quad x>-\frac12.
\end{equation}
Set $x=m-1/2$; this gives
\begin{equation}\label{pf4wim9}
\sum_{k = 0}^\infty  {\binom{{m - 1/2}}{k}^4 \left( {4k - 2m + 1} \right)}  = \frac{{( - 1)^m 2^{4m} }}{{\pi ^2 m}}\binom{{2m}}{m}^{ - 1} ,
\end{equation}
after using
\begin{equation}
\sin \left( {\frac{\pi }{2}\left( {2m - 1} \right)} \right) = ( - 1)^{m - 1}
\end{equation}
and
\begin{equation}
\binom{{2m - 1}}{{m - 1/2}} = \frac{{\Gamma (2m)}}{{\Gamma ^2 \left( {m + 1/2} \right)}} = \frac{{2^{4m - 1} }}{{\pi m}}\binom{{2m}}{m}^{ - 1} .
\end{equation}
Identity~\eqref{ua6fod4} now follows from~\eqref{pf4wim9} upon application of~\eqref{hk5w16y}.
\end{proof}

\begin{theorem}
If $m$ is a non-negative integer, then
\begin{equation}\label{f9cme1q}
\sum_{k = 0}^\infty  {\left( {\frac{1}{{2^{2k} }}\binom{{2k}}{k}\prod_{j = 1}^m {\frac{1}{{2k - 2j + 1}}} } \right)^3 }  = \left( {\frac{{m!}}{{2^m }}\binom{{2m}}{m}} \right)^{ - 3} \frac{{( - 1)^{\left\lceil {m/2} \right\rceil } 24\sqrt 2 }}{{\left( {2m - 1} \right)^2 }}\frac{{\Gamma \left( {\frac{3}{4}\left( {2m - 1} \right)} \right)}}{{\Gamma ^3 \left( {\frac{1}{4}\left( {2m - 1} \right)} \right)}}.
\end{equation}
\end{theorem}
\begin{proof}
Write~\eqref{gfrhf9n} as
\begin{equation}
\sum_{k = 0}^\infty  {( - 1)^k \binom{{x}}{k}^3 }  = \frac{{\Gamma \left( {\frac{1}{2}\left( {3x + 2} \right)} \right)}}{{\Gamma ^3 \left( {\frac{1}{2}\left( {x + 2} \right)} \right)}}\cos \left( {\frac{{\pi x}}{2}} \right),
\end{equation}
and set $x=m-1/2$ to obtain
\begin{equation}
\sum_{k = 0}^\infty  {( - 1)^k \binom{{m - 1/2}}{k}^3 }  = \frac{{48}}{{\left( {2m - 1} \right)^2 }}\frac{{\Gamma \left( {\frac{3}{4}\left( {2m - 1} \right)} \right)}}{{\Gamma ^3 \left( {\frac{1}{4}\left( {2m - 1} \right)} \right)}}\cos \left( {\frac{\pi }{4}\left( {2m - 1} \right)} \right);
\end{equation}
and hence~\eqref{f9cme1q}.
\end{proof}
\begin{corollary}
If $m$ is a non-negative integer, then
\begin{align}
&\sum_{k = 0}^\infty  {\left( {\frac{1}{{2^{2k} }}\binom{{2k}}{k}\prod_{j = 1}^{2m} {\frac{1}{{2k - 2j + 1}}} } \right)^3 }\\
&\qquad  = \left( {\frac{{(2m)!}}{{2^{2m} }}\binom{{4m}}{{2m}}} \right)^{ - 3} \frac{{( - 1)^m }}{{\left( {4m - 1} \right)^2 }}\frac{{\Gamma ^4 (1/4)}}{{4\pi ^3 }}\frac{{\prod_{j = 0}^{3m - 1} {\left( {4j - 3} \right)} }}{{\prod_{j = 0}^{m - 1} {\left( {4j - 1} \right)^3 } }},
\end{align}
and
\begin{align}
&\sum_{k = 0}^\infty  {\left( {\frac{1}{{2^{2k} }}\binom{{2k}}{k}\prod_{j = 1}^{2m + 1} {\frac{1}{{2k - 2j + 1}}} } \right)^3 }\\
&\qquad  = \left( {\frac{{(2m + 1)!}}{{2^{2m + 1} }}\binom{{2\left( {2m + 1} \right)}}{{2m + 1}}} \right)^{ - 3} \frac{{( - 1)^{m + 1} }}{{\left( {4m + 1} \right)^2 }}\frac{{48\pi }}{{\Gamma ^4 (1/4)}}\frac{{\prod_{j = 1}^{3m} {\left( {4j - 1} \right)} }}{{\prod_{j = 1}^m {\left( {4j - 3} \right)^3 } }}.
\end{align}
In particular,
\begin{align}
\sum_{k = 0}^\infty  {\left( {\frac{1}{{2^{2k} }}\binom{{2k}}{k}} \right)^3 }  &= \frac{{\Gamma ^4 \left( {1/4} \right)}}{{4\pi ^3 }},\label{f3uy98b}\\
\sum_{k = 0}^\infty  {\left( {\frac{1}{{2^{2k} \left( {2k - 1} \right)\left( {2k - 3} \right)}}\binom{{2k}}{k}} \right)^3 }  &=  - \frac{5}{{324}}\frac{{\Gamma ^4 \left( {1/4} \right)}}{{\pi ^3 }},\\
\sum_{k = 0}^\infty  {\left( {\frac{1}{{2^{2k} \left( {2k - 1} \right)}}\binom{{2k}}{k}} \right)^3 }  &=  - \frac{{48\pi }}{{\Gamma ^4 \left( {1/4} \right)}},
\end{align}
and
\begin{equation}
\sum_{k = 0}^\infty  {\left( {\frac{1}{{2^{2k} \left( {2k - 1} \right)\left( {2k - 3} \right)\left( {2k - 5} \right)}}\binom{{2k}}{k}} \right)^3 }  = \frac{{1232}}{{9375}}\,\frac{\pi }{{\Gamma ^4 \left( {1/4} \right)}}.
\end{equation}
\end{corollary}
\begin{remark}
Identity~\eqref{f3uy98b} was recorded by Chen~\cite{chen22}.
\end{remark}

\end{document}